\documentclass[twoside,11pt]{amsart} 

\usepackage[utf8]{inputenc}
\usepackage{amsmath}
\usepackage{amsfonts}
\usepackage{amssymb}

\usepackage{thmtools}

\usepackage{cancel}
\usepackage{graphicx}

\usepackage{multicol,array}

\usepackage{tikz}
\usetikzlibrary{matrix}

\newtheorem{question}{Question}
\newtheorem{thm}{Theorem}[section]
\newtheorem{definition}[thm]{Definition}

\newtheorem{remark}[thm]{Remark}

\newtheorem{example}[thm]{Example}

\newtheorem{corollary}[thm]{Corollary}

\newcommand{\ignore}[1]{}

\newcommand{\FrakA}{\mathfrak{A}}
\newcommand{\FrakB}{\mathfrak{B}}

\newcommand{\A}{\mathcal{A}}
\newcommand{\B}{\mathcal{B}}
\newcommand{\C}{\mathcal{C}}

\newcommand{\CalL}{\mathcal{L}}
\newcommand{\M}{\mathcal{M}}

\renewcommand{\S}{\mathcal{S}}

\newcommand{\structs}{\mathbb{E}}
\newcommand{\natnum}{\omega}

\newcommand{\pair}[1]{\langle #1 \rangle}
\newcommand{\Text}{\mathbf{Txt}}
\newcommand{\Inf}{\mathbf{Inf}}

\newcommand{\Ex}{\mathbf{Ex}}
\newcommand{\Bc}{\mathbf{Bc}}
\newcommand{\Fin}{\mathbf{Fin}}

\newcommand{\isom}{\cong}

\newcommand{\embeds}{\hookrightarrow}
\newcommand{\embedsFin}{\hookrightarrow_{\mathrm{fin}}}

\newcommand{\biEmbeds}{\approx}
\newcommand{\biEmbedsFin}{\approx_{\mathrm{fin}}}

\newcommand{\set}[1]{\{{#1}\}}

\DeclareMathOperator{\dom}{dom}

\DeclareMathOperator{\range}{range}

\DeclareMathOperator{\chara}{char}

\DeclareMathOperator{\SEP}{SEP}
\DeclareMathOperator{\sep}{sep}

\DeclareMathOperator{\content}{content}

\title[Limit Learning Equivalence Structures]{Limit Learning Equivalence Structures}

\author[E.~Fokina]{Ekaterina Fokina}
\address{Institute of Discrete Mathematics and Geometry, Vienna University
of Technology, Vienna, Austria}
\email{ekaterina.fokina@tuwien.ac.at}

\author[T.~K{\"o}tzing]{Timo K{\"o}tzing}
\address{Hasso Plattner Institute,
University of Potsdam, Germany
}
\email{timo.koetzing@hpi.de}

\author[L.~San Mauro]{Luca San Mauro}
\address{Institute of Discrete Mathematics and Geometry, Vienna University
of Technology, Vienna, Austria}
\email{luca.san.mauro@tuwien.ac.at}

\thanks{The authors are grateful to the referees for their thorough comments. Fokina was supported by the Austrian Science Fund, project~P~27527.  San Mauro was supported by the Austrian Science Fund FWF, project~M~2461. K{\"o}tzing was supported by the German Research Foundation (DFG) under grant number KO 4635/1-1.}


\begin{document}

\maketitle

\begin{abstract}
While most research in Gold-style learning focuses on learning formal languages, we consider the identification of computable structures, specifically equivalence structures. In our core model the learner gets more and more information about which pairs of elements of a structure are related and which are not. The aim of the learner is to find (an effective description of) the isomorphism type of the structure presented in the limit. In accordance with language learning we call this learning criterion $\Inf\Ex$-learning (explanatory learning from informant).

Our main contribution is a complete characterization of which families of equivalence structures are  $\Inf\Ex$-learnable. This characterization allows us to derive a bound of $\mathbf{0''}$ on the computational complexity required to learn uniformly enumerable families  of equivalence structures. We also investigate variants of $\Inf\Ex$-learning, including learning from text (where the only information provided is which elements are related, and not which elements are not related) and finite learning (where the first actual conjecture of the learner has to be correct). 
Finally, we show how learning families of structures relates to learning classes of languages by mapping learning tasks for structures to equivalent learning tasks for languages.
\end{abstract}

\section{Introduction}

Consider a learner observing (a countably infinite number of) different items to be equivalent or not equivalent. The learner would like to arrive at a conjecture about the \emph{structure} of this equivalence relation, that is, the learner would like to determine the \emph{isomorphism type} of the equivalence structure embodied by the items. For example, the learner could see more and more groups of $5$ equivalent objects, and no groups of other sizes, and announces as a conjecture ``infinitely many equivalence classes of size $5$ and no equivalence classes of other sizes''. If the first guess has to be correct, we call the setting \emph{finite learning} (denoted $\Fin$), if the conjecture may be changed an arbitrary (but finite) number of times before stabilizing on a correct conjecture, we call the setting \emph{explanatory learning} (denoted $\Ex$). In each case, the data available to the learner is a \emph{complete} \emph{accurate} list of which elements of the structure are equivalent and which are not. Following standard convention in inductive inference, we call this \emph{learning from informant} ($\Inf$), where both positive and negative information is available.

In general, this style of learning is called \emph{learning in the limit} or \emph{inductive inference}, and dates back to Gold~\cite{Gol:j:67}. Most work in inductive inference concerns either learning of formal languages or learning of total functions (see the text books \cite{Osh-Sto-Wei:b:86:stl,Jai-Osh-Roy-Sha:b:99:stl2}), the case of learning other structures has first been considered by Glymour~\cite{Gly:j:85} and is surveyed by Martin and Osherson~\cite{Mar-Osh:b:98}. More recently, researchers investigated the case in which the languages to be learned correspond to substructures of a given structure. For instance, Stephan and Ventsov~\cite{Ste-Ven:j:01:algebraic}, Harizanov and Stephan~\cite{HarizanovS07}, and Merkle and Stephan~\cite{stephan2004} considered learnable ideals of rings, subgroups
and submonoids of groups, subspaces of vector spaces and isolated branches on uniformly computable
sequences of trees. They showed that different types of learnability of various families of
computable or computably enumerable structures can be characterized algebraically.

With the present paper, we want to strengthen the connection between algorithmic learning theory and computable structure theory by developing a learning framework in which one can formalize the intuition of learning an arbitrary structure in the limit (the interested reader can consult \cite{Ash-Knight:Book} for a classical introduction to computable structure theory). To this end, equivalence structures represent an ideal case-study. Although being fairly basic from an algebraic point of view, equivalence structures exhibit many deep effective properties, and thus they attracted much attention from computable theorists. We offer just few examples: Calvert, Cenzer, Harizanov, and Morozov~\cite{calvert2006effective} classified computable equivalence structures that possess a unique computable presentation up to computable isomorphism;
Downey, Melnikov, and Ng~\cite{downey2017friedberg} studied the complexity of listing computable equivalence structures with no repetitions; and recently there has been
an increasing interest in analyzing the effective content of computably enumerable equivalent relations, as in \cite{Andrews:14}.

\smallskip

In what follows,  we denote by $\omega$ the set of natural numbers and we use a convenient short hand to denote isomorphism types.  For any function $f: \natnum \cup \{\omega\} \rightarrow \natnum \cup \{\omega\}$ we denote by $[f]$  the isomorphism type of exactly $f(a)$ many equivalence classes of size $a$, and if only finitely many values of $f$ are non-zero, then we can list all these values as $[n_0:f(n_0),\ldots, n_k:f(n_k)]$. For example, $[5:\omega]$ denotes the isomorphism type of infinitely many equivalence classes of size $5$ and no others. Note that, to learn a structure $\A$, we must learn any presentation of $\A$ with members of the natural numbers.

Trivially, a single structure, or also a single isomorphism type, is always learnable by a learner which constantly outputs a correct conjecture. Thus, we are interested in which \emph{families} of structures can be learned by a single learner. Thus, we consider families $\FrakA$ of equivalence relations on $\natnum$ and ask whether there is a single learner $M$ such that $M$ can learn any $\A \in \FrakA$ when given more and more (accurate and, in the limit, complete) information about $\A$. We will mostly consider arbitrary functions as learners, but  we will also discuss the computational complexity of learning. The following example illustrates the concept of learning families of structures.

\begin{example}
$[5:\omega, 6:2]$ and $[5:\omega,7:1]$ are simultaneously $\Inf\Fin$-learnable (finitely learnable from informant): If there is ever an equivalence class of size $7$ in the input, then the second structure is the only possibility, while once there are two classes of size $6$ are in the input, the first structure is the only possibility.
\end{example}

It is most interesting to see what parts of the learning setting influence the learning power, and in what way. For example, we might wonder whether the ability to change the hypothesis arbitrarily often (as in explanatory learning) gives an advantage over finite learning. The next example shows that the learning power of finite learning is indeed smaller than that of explanatory learning.
\begin{example}
$[5:\omega]$, $[6:\omega]$ are not simultaneously $\Inf\Fin$-learnable, but $\Inf\Ex$-learnable: Regarding the negative part, any finite information about an instance of $[5:\omega]$ can be extended to an instance of $[6:\omega]$, so at no time can the learner commit to a hypothesis. An explanatory learner on the other hand can conjecture $[5:\omega]$ until any equivalence class of size $6$ appears in the input and then change to $[6:\omega]$.
\end{example}

This shows that $\Inf\Fin$-learning is unreasonably weak, only very restricted families of structures are learnable in this sense (we characterize finite learning in two ways in Theorem~\ref{thm:FinLearningOfFiniteClasses}). Note that a common way of restricting the learner even less than in $\Inf\Ex$-learning is by not requiring \emph{syntactic} convergence to a final hypothesis, but only \emph{semantic} convergence; that is, from some point on, all conjectures are correct, just not necessarily the same. The corresponding learning criteria replace $\Ex$ by $\Bc$ (behaviorally correct). We are mostly interested in learnability by arbitrary learners which can check for equivalence of conjectures (which is typically undecidable), so for our setting this relaxation does not make a difference. Thus, in this paper, we are interested in understanding the $\Inf\Ex$ learning criterion.

 So far we have seen examples of families that are $\Inf\Ex$-learnable. In order to establish that some families are not learnable, we turn to the concept of \emph{locking sequences}, which is used extensively to show nonlearnability in the setting of learning formal languages. A locking sequence is a sequence of inputs $\sigma$ for the target concept such that the learner does not change its mind regardless of how $\sigma$ is extended with information for the target concept. Since we only want to learn up to isomorphism and the original concept of locking sequences is adjusted to exact learning, we get a slightly different notion of locking sequence (see the appendix for details). With this we can get the following result: there are two structures which are bi-embeddable (i.e., there is an embedding from any of the two structures into the other), but not simultaneously learnable. Clearly, if we only required learning up to bi-embeddability, then the two structures would be simultaneously learnable (by a constant learner). We use $\Inf\Ex_{\isom}$ to denote learning up to ismorphism and $\Inf\Ex_{\biEmbeds}$ for learning up to bi-embeddability\footnote{We use $\isom$ to denote isomorphism of structures and $\biEmbeds$ for bi-embeddability of structures.}.  This result is summarized in the following example.
\begin{example}\label{ex:biEmbedsNotIsom}
$[5:\omega, 2:1]$, $[5:\omega]$ are not simultaneously $\Inf\Ex_{\isom}$-learnable, but $\Inf\Ex_{\biEmbeds}$-learnable: Since the structures are bi-embeddable, a constant learner can $\Inf\Ex_{\biEmbeds}$-learn both. Proving that these two structures are not $\Inf\Ex_{\isom}$-learnable is a bit harder: intuitively, the problem is that any finite fragment of the two structures can be extended into a finite fragment of the other, in such way that a potential learner would be forced to change its mind infinitely often (see Theorem~\ref{thm:characterization} and Theorem~\ref{thm:computablyBiembedButNotIsomLearnableShort}).
\end{example}
Interestingly, Example~\ref{ex:biEmbedsNotIsom} shows that there is a class containing only finitely many different learning targets (in this case two), but it is still not $\Inf\Ex_{\isom}$-learnable; this is in contrast to language learning, where learning finitely many different learning targets can always be distinguished.

Learning up to bi-embeddability is also of independent interest. In recent years, the relation of bi-embeddability received much attention in descriptive set theory and computable structure theory. In particular, bi-embeddability is of fundamental interest for classifying the complexity of equivalence relations in terms of Borel reducibility, as in \cite{friedman2011analytic}, and the effective countepart of Borel reducibility discussed, e.g., in \cite{fokina2012isomorphism}. See also \cite{fokina_comp-bi-embeddability_2017} for a full classification of computable presentations of equivalence structures up to bi-embeddability.

In our examples so far we have never exploited the information that two elements are not related (the negative information). Learning without negative information is called \emph{learning from text} (as opposed to learning from informant) and is denoted by $\Text$ instead of $\Inf$. In fact we can show that, for learning structures neither of which has an equivalence class of size $\omega$, informant and text learning are equivalent (see Theorem~\ref{thm:textEquivalentToInformant}). The following example shows that this does not extend to structures which contain equivalence classes of size $\omega$.
\begin{example}
$[\omega:1]$, $[\omega:2]$ is $\Inf\Ex_{\isom}$-learnable, but not $\Text\Ex_{\isom}$-learnable: Regarding the positive part, conjecture $[\omega:1]$ until two elements are known to not be in the same equivalence class, then conjecture $[\omega:2]$. The negative part is based on the concept of locking sequences, see Theorem~\ref{thm:textNoteEquivalentToInformant}.
\end{example}

These examples already give a good impression of what is learnable and what is not learnable with which kind of strategies. To further extend our intuition on what is $\Inf\Ex_{\isom}$-learnable and what is not, we consider the following example.
\begin{example}\label{ex:notLearningLimitInstances}
The infinite class of structures $\set{[5:n, 1:\omega] \mid n \in \natnum} \cup \set{[5:\omega]}$ is \emph{not} $\Inf\Ex_{\isom}$-learnable: Intuitively, when a learner tries to learn $[5:\omega]$ and stops making mind changes after having seen some finite number $n$ of equivalence classes of size $k$, the learner cannot successfully learn $[5:n, 1:\omega]$, since any extension to a finite part of $[5:n, 1:\omega]$ can be extended to $[5:\omega]$.
\end{example}

We use these intuitions in Section~\ref{sec:characterizing} to give a  characterization of $\Inf\Ex_{\isom}$-learning which we call \emph{finite separability}. 
Intuitively, a class of structures $\FrakA$ is finitely separable if nonisomorphic structures of $\FrakA$ that are finitely bi-embeddable can  be distinguished by looking at  some finite fragment of them. 
This characterization now completely informs about which families of structures are learnable. This simplifies proofs and furthermore allows us to give a bound on the complexity of learning c.e.\ enumerations of structures (see Section~\ref{sec:complexityOfLearning}). Note that the notion of finite separability is similar to the existence of \emph{tell-tales} as used for characterizing learnable classes of languages(see \cite{Ang:j:80:lang-pos-data}).

With this characterization we were able to approach the complexity of $\Inf\Ex_{\isom}$-lear\-ning and show that, for uniformly enumerable sets of $\Inf\Ex_{\isom}$-learnable structures, $\mathbf{0''}$-computable learners are sufficient for learning, while com\-pu\-ta\-ble learners are not. 

For the reader familiar with language learning we provide two embeddings of learning of equivalence structures into the setting of language learning in Section~\ref{sec:languageLearning}. For both $\Inf\Ex_{\isom}$ and $\Inf\Ex_{\biEmbeds}$ we can map families of structures in an intuitive way to classes of languages such that a class of structures can be learned iff its image under the corresponding map is $\Text\Ex$-language-learnable, provided that the structure of learning equivalence classes is in some sense a substructure of learning languages from text.

Note that, compared with language learning, learning of structures provides interesting new settings in which targets do not have to be learned exactly, but only up to some equivalence relation on structures (such as isomorphism). Furthermore, learning up to isomorphism has the advantage that ``coding tricks'' from language learning (making classes of languages learnable by having each language ``encode'' a correct hypothesis artificially in the data) are somewhat avoided. See \cite[§13]{Jai-Osh-Roy-Sha:b:99:stl2} for a discussion on coding tricks.

\section{Learning of Structures}

Our object of study is the class $\structs$ of the equivalence structures $\A$ on natural numbers. For the benefit of exposition, we assume that all our equivalence structures are of the form $(\omega,E)$, where $E$ is an equivalence relation on $\omega$. We say that $\A$ is an \emph{$\omega$-presentation} of $\M$ if $\A \isom \M$ and $\A$ has domain $\omega$. Note that the choice of limiting our focus to $\omega$-presentations of equivalence structures is not a strong restriction: given any infinite equivalence structure $\M = (M,E)$, with $M = \set{m_0,m_1,\ldots}$, one can use the bijection $i \mapsto m_i$ to get an equivalence structure $\A$ that is an $\omega$-presentation of $\M$.

The \emph{atomic diagram} of $\A$ is the set of atomic formulas and negations of atomic formulas true of $\A$.

An equivalence structure $\A$ is \emph{computably presentable} if $\A$ is isomorphic to a computable equivalence structure. 
To formally define our learning framework, we rely on some effective enumeration of the computable structures from $\structs$, up to isomorphism. This can be done in many ways; for instance, Downey, Melnikov, and Ng~\cite{downey2017friedberg} showed, with a rather involved proof, that such an enumeration can be constructed which is a \emph{Friedberg} enumeration, i.e., with no repetitions of the isomorphism types. For our interests, it is enough to fix some computable  sequence of equivalence relations
$(E_i )_{i\in\omega}$, where  $E_i$ is an equivalence relation on $\omega$, such that any infinite computable equivalence structure is isomorphic to $\M_i=(\omega,E_i)$, for some $i$ (see \cite{downey2017friedberg} for more details). We say that $e$ is a \emph{conjecture} for $\M_e$.

Recall that we aim at modeling a learner that receives larger and larger finite pieces of information about some equivalence structure $\A$.

A \emph{text} is a function $T: \natnum \rightarrow \natnum^2 \cup \set{\#}$, where $\#$ is a special symbol denoting a \emph{pause}, that is, no new information. We let $\content(T) = \range (T) \setminus \set{\#}$ be the \emph{content} of $T$. For any text $T$ and equivalence structure $\A = (\omega,E)$, we say \emph{$T$ is a text for $\A$} iff $\content(T) = E$ (that is, $\content(T)$ is all and only the positive information about $\A$). Note that pause symbol is required so that the  structure $[1:\omega]$, where no element is related to any other, has a text. By $\Text(\A)$ we denote the set of all texts for $\A$.

An \emph{informant} is a function $I: \natnum \rightarrow \natnum^2 \times \set{0,1}$ such that, for any $(x,y) \in \natnum^2$, either $((x,y),0) \in \range(I)$ or $((x,y),1) \in \range(I)$ (but never both). We let $\content^+(T) = \set{(x,y) \mid ((x,y),1) \in \range(I)}$ be the \emph{positive content} of $I$. Intuitively, an informant eventually lists, for each pair of elements, whether they are related or whether they are unrelated. For any informant $I$ and equivalence structure $\A = (\omega,E)$, we say \emph{$I$ is an informant for $\A$} iff $\content^+(T) = E$. By $\Inf(A)$ we denote the set of all informants for $\A$.

For any function $f$ defined on natural numbers (such as texts and informants) and $n \in \natnum$ we let $f[n]$ denote the finite sequence $f(0),\ldots,f(n-1)$.

A learner is a function mapping initial segments of texts or informants to conjectures (elements of $\natnum \cup \set{?}$). The \emph{learning sequence} of a learner $M$ on a text or informant $f$ is $p:\natnum \rightarrow \natnum \cup \set{?}$  such that $p(n)=M(f[n])$.

For any finite sequence $\sigma$ which is an initial part of an informant, we let $\A_\sigma$ be the finite structure encoded by $\sigma$ by using as universe all elements mentioned either positively or negatively in $\sigma$, taking the transitive closure of all positively mentioned pairs and assuming all other relations to be negative. We denote by $\A[s]$ the finite substructure of $\A$ with domain $\set{0,\ldots,s-1}$.

Any predicate on learning sequences and $\structs$ is called a \emph{learning restriction}. We use the following learning restrictions.
\begin{definition}\label{def:exlearn}
Let $\sim$ be any equivalence relation on structures.\footnote{The equivalence relation $\sim$ intuitively defines that $[\A]_{\sim}$ is the target at which a learner $M$ is supposed to aim (typically $\sim$ is $\isom$ or $\biEmbeds$).}
We define the learning restriction $\Ex_\sim$ on a learning sequence $p$ and $\A \in \structs$ such that
$$
\Ex_\sim(p,\A) \Leftrightarrow \exists e \forall^\infty n: p(n) = e \wedge \A \sim \M_e.
$$
Further, we define the learning restriction corresponding to finite learning by 
$$
\Fin_\sim(p,\A) \Leftrightarrow \exists e: \set{e} \subseteq \range(p) \subseteq \set{e,?} \wedge \A \sim \M_e.
$$
\end{definition}

\begin{definition}
A \emph{learning criterion} is a triple of a set  $\C$  of (partial) functions $\natnum \rightarrow \natnum$ (the set of admissible learners), an operator $\alpha$ returning, for a given structure, a set of presentations for that structure (either all texts or all informants) and a learning restriction $\delta$. We also write $\C$-$\alpha\delta$ for the learning criterion $(\C,\alpha,\delta)$. The class of \emph{$\C$-$\alpha\delta$-learnable} structures contains all those sets $\FrakA$ of structures such that there is a learner $M \in \C$ such that, for all $\A \in \FrakA$, all $\omega$-presentations $\A^*$ of $\A$, and all $f \in \alpha(\A^*)$, $\delta(n \mapsto M(f[n]), \A^*)$. 
\end{definition}

Note, in the above definition, that to learn a structure $\A$, a learner should learn all the $\omega$-presentations of $\A$. Sometimes we will denote by $M(\A)$ the limit conjecture (if exists) of the learner $M$ on input $\A$ and by $M(\mathcal A[s])$ the conjecture $M$ produces when given a string encoding $\mathcal A[s]$.

\subsection{Notation Regarding Equivalence Structures}\label{notation}
The \emph{character} $char(\A)$ of $\A$ is
\[ 
\chara({\A})=\set{\langle k,i\rangle: \A \mbox{ has at least $i$ equivalence classes of size $k$, for $k\in\omega\cup\set{\omega}$}}.
\]
We call any element of $char({\A})$ a \emph{component} of $\A$. Sometimes, we will  approximate the character of $\A$ as follows
\[
\chara({\A}[s])=\set{\langle k,i\rangle: \A[s] \mbox{ has at least $i$ equivalence classes of size $k$, for $k\in\omega$}}.
\]

\smallskip

Let $\A=(\omega,E_{\A})$ and $\B=(\omega,E_{\B})$ be in $\structs$.  $\A$ \emph{embeds} into $\B$ (notation: $\A \embeds\B$)  if there is a injection $f: \omega \rightarrow \omega$ such that, for all $i,j \in \omega$, $i{E_\A}j \Leftrightarrow f(i){E_{\B}}f(j)$.  $\A$ \emph{finitely embeds} into $\B$ (notation: $\A \embedsFin \B$) if $\A[s]\embeds \B$, for all $s\in\omega$. $\A$ and $\B$ are \emph{bi-embeddable} (notation: $A \approx B$) if they embeds in each other. $\A$ and $\B$ are \emph{finitely bi-embeddable} (notation: $A \biEmbedsFin B$) if they finitely embeds in each other. $\A$ and $\B$ are \emph{isomorphic} (notation: $A \cong B$) if $\A \embeds \B$ via a bijection $f: \omega \rightarrow \omega$.

\section{Characterizing $\Inf\Ex_{\isom}$}
\label{sec:characterizing}

Our main focus is on the class $\Inf\Ex_{\isom}$. To help the reader get acquainted with our framework, we stress that some $\FrakA\in \Inf\Ex_{\isom}$ if there is a learner $M$ (of arbitrary complexity) such that, for any $\omega$-presentation $\A^*$ of a structure $\A\in\FrakA$, $M(\A^*)\cong \A$.  In this section we characterize which families of equivalence structures are $\Inf\Ex_{\isom}$-learnable. For the ease of presentation, we focus on equivalence structures with no infinite classes. Yet, it is not hard to modify the forthcoming analysis in order to obtain a full characterization of $\Inf\Ex_{\isom}$; this will be done in future work.

\begin{definition}\label{def:finiteseparability}
Let $\FrakA$ be a family of equivalence structures with no infinite classes, and $\mathcal{S}$ an equivalence structure. 

\begin{enumerate}
\item If $\FrakA$ has finitely many isomorphism types, then $\mathcal{S}$ is a \emph{limit} of $\FrakA$ if there is $\A \in \FrakA$ such that 
\[
\A\not\cong \S \wedge \A \embedsFin \S,
\]
\[
\mbox{and }\chara(\S) \subseteq \chara(\A).
\]

\item If $\FrakA$ has infinitely many isomorphism types, $\S$ is a \emph{limit} of $\FrakA$ if

\[
(\forall \A \in \FrakA)(\A \not\cong\S \wedge \A \embedsFin \S)
\]
\[
\mbox{and }\chara({\S})\subseteq \set{\langle k,i\rangle : \FrakA \mbox{ contains infinitely many $\B$'s  with } \langle k,i\rangle \in \chara(\B)}.
\]
\end{enumerate}

 $\FrakA$ is \emph{finitely separable} if, for all $\FrakB\subseteq \FrakA$, $\FrakB$ has no limits in $\FrakA$.
\end{definition}

To clarify the above notion of limit, together with the corresponding one of finite separability, let us compare it with known examples of failure of $\Inf\Ex_{\cong}$-learning.

Condition $1.$ is designed to deal with cases such as Example \ref{ex:biEmbedsNotIsom}: It says that $\S$ cannot be finitely separated from $\A$ if $\A$ finitely embeds in $\S$ and any component of $\S$ is a component of $\A$. Intuitively, this makes impossible to $\Inf\Ex_{\cong}$-learn $\set{\A,\S}$ because we can build an $\omega$-presentation $\S^*$ of $\S$ such that $\S^*$ has arbitrarily large fragments  that resemble $\A$, and this forces any potential learner $M$ to have infinitely many mind changes if attempting to learn $\S^*$.

The infinite case is handled by Condition $2.$ and turns out to be more delicate. But the idea is the same and it aims at formalizing cases such as Example \ref{ex:notLearningLimitInstances}: If $\S$ is the limit of an infinite family of pairwise nonisomorphic structures $\FrakA$, then we can construct an $\omega$-presentation of $\S$ that, for infinitely many initial segments, looks like some structure of $\FrakA$. In doing so, we eventually obtain a structure isomorphic to $\S$ because each component of $\S$ is witnessed by  infinitely many structures of $\FrakA$. 

\smallskip

We make the latter observations more precise by proving that finite separability coincides with $\Inf\Ex_{\cong}$-learnability.

\begin{remark}
In the following proof, we introduce the formal notion of separators (of a given structure) that justify the terminology ``finite separability'' and the underlying intuition that, if $\FrakA$ is finitely separable, then structures from $\FrakA$ can be distinguished by looking at only finite fragments of them.
\end{remark}

\begin{thm}\label{thm:characterization}
Let $\FrakA$ be a family of equivalence structures with no infinite classes. We have that
\[
 \FrakA \mbox{ is finitely separable} \Leftrightarrow \FrakA \in \Inf\Ex_{\cong}.
\]
\end{thm}

\begin{proof}
$(\Rightarrow):$ Assume that $\FrakA=\set{\A_i}_{i\in\omega}$ is finitely separable.  To show that $\FrakA \in \Inf\Ex_{\cong}$, we first prove that there is a learner $M$ that learns the $\biEmbedsFin$-type of any given $\A \in \FrakA$. Assume that 
$M$ reads $\A_z$ and, for all stages $s$, let $\FrakB_s\subseteq \FrakA$ be the class of equivalence structures in which $\A_z[s]$ is finitely embeddable and that are minimal with respect to $\embedsFin$. We define $M$ as follows: at stage $s$, $M$ outputs the $\biEmbedsFin$-type  of the equivalence structure in $\FrakB_s$ with least index in the enumeration $\set{\A_i}_{i\in\omega}$. 

\smallskip

We claim that via this procedure $M$ learns the $\biEmbedsFin$-type of $\A_z$. In particular, we prove  that $[\A_z]_{\biEmbedsFin}\supseteq \bigcap_{i\in \omega}{\FrakB_i}$.,

It is clear that $\A_z \in \bigcap_{i\in \omega}{\FrakB_i}$. Towards a contradiction, suppose that there is $A_w \not\biEmbedsFin A_z$ such that $\A_w \in \bigcap_{i\in \omega}{\FrakB_i}$. We distinguish three cases. 
\begin{enumerate}
\item[a)] If $A_w$ and $A_z$ are $\embedsFin$-incomparable, then there must be a component of $\A_z$ that is not a component of $\A_w$. So, there
exists a stage $s$ such that $\A_z[s]$ does not embed in $A_w$. Therefore, $A_w \notin \bigcap_{i\geq s}{\FrakB_i}$.
\item[b)] If $\A_w \embedsFin \A_z$, then there must be a component of $\A_z$ witnessing the fact that $\A_z\not\embedsFin A_w$. So,  there
exists again a stage $s$ such that $\A_z[s]$ does not embed in $A_w$. Therefore, $\A_w$ will be eventually outside from the $\FrakB_s$'s.
\item[c)] If $\A_z\embedsFin \A_w$, then, for all $s$, $\A_w\in \FrakB_s$ only if $\A_z \notin \FrakB_s$. This is because $\FrakB_s$ contains only structures that are minimal with respect to\ $\embedsFin$. But we already know that $\A_z \in \bigcap_{i\in \omega}{\FrakB_i}$. Therefore, $A_w \notin \bigcap_{i\in \omega}{\FrakB_i}$.
\end{enumerate}

This shows that $M$ correctly learns the $\biEmbedsFin$-type of any given  $\A_{z}$. Call $j$ the limit conjecture of $M$. 

\smallskip

Now we construct a learner $M^*$ that learns the isomorphism type of $\A_z$.
To choose among the structures in $[\M_j]_{\biEmbedsFin}$, $M^*$ adopts the following procedure: For $\C\in [\M_j]_{\biEmbedsFin}$, let $\sep(\C)$ be the \emph{separator of $\C$} defined as 
\[
\sep(\C)= \bigcup_{\S \in [\M_j]_{\biEmbedsFin}}\set{\min(\chara(\C) \setminus \chara(\S))}.
\]

First, note that separators of pairwise nonisomorphic structures form an anti-chain with respect to $\subseteq$. This follows from the fact that $\chara(\C) \setminus \chara(\S) \neq \emptyset$, for all $\C \not\cong\S$ in $[\M_j]_{\biEmbedsFin}$.  Otherwise, $\C$ would be a limit of $\set{\S}$, since $\chara({\C})\subseteq \chara({\S})$ and $\S \embedsFin \C$, and this would contradict the finite separability of $\FrakA$.

Second, we claim that separators are all finite. Towards a contradiction assume that, for given $\C\in [\M_j]_{\biEmbedsFin}$, $\sep({\C})$ is infinite. Denote by $\chara(\C)\upharpoonright_i$ the finite set consisting of the first $i$ elements of $\chara(\C)$. Let 
\[\mathfrak{S}_{i}=\set{ \S\not\cong \C \in [\M_j]_{\biEmbedsFin} : \chara(\C)\upharpoonright_i \;\subseteq  \chara({\S})}.
\] 
We have that $\mathfrak{S}_{i}$ is nonempty, for all $i$. Otherwise, if $k$ is the least such $\mathfrak{S}_{k}=\emptyset$, we would obtain that $\sep(\C)\subseteq \chara(\C)\upharpoonright_k$, against the hypothesis that $\sep(\C)$ is infinite.  
Suppose that $\bigcap_{i\in\omega}\mathfrak{S}_{i}$ is nonempty and contains a structure $\mathcal{C}^*$. It immediate to see that $\chara({\C})\subseteq \chara({\C^*})$. Recall that $\C^*\embedsFin \C$, since $\C^*\biEmbedsFin\C$. It follows that $\C$ is a limit of $\set{\C^*}$, which is impossible because of the finite separability of $\FrakA$.

But if all the $\mathfrak{S}_{i}$'s are nonempty, it follows that 
\[
\chara(\C)\subseteq \bigcup_{i\in\omega}\bigcap_{\S \in \mathfrak{S}_i}\chara(\S),
\]

and this contradicts  the finite separability of $\FrakA$.  Thus, we conclude that, for all $\C\in [\M_j]_{\biEmbedsFin}$,  $\sep(\C)$ is finite.

\smallskip

We aim at making use of the finiteness of separators to show that $M^*$ can eventually learn the isomorphism type of $\A_z$. We say that a finite equivalence structure $\S[t]$ \emph{realizes} a given separator $\sep({\S})$ if $\sep({\S})\subseteq \chara({\S[t]})$. Denote by $\SEP(\S[t])$ the class of the separators realized by $\S[t]$.  Note that  $\A_z$ eventually realizes its own separator, i.e., there is $s$ such that, for all $t\geq s$, $\sep({\A_z}) \in \SEP(\A_z[t])$. The problem is that $\SEP(\A_z[t])$ might consist of more than one separator for infinitely many stages, i.e., for infinitely many $t$, there might be $\A_w \not\cong\A_z$ and $\A_w \biEmbedsFin\A_z$ such that $\sep({\A_w})\in \SEP(\A_z[t])$.
Nonetheless, if any such $\sep({\A_w})$ is realized by some $\A_z[t]$, there must be $t^*>t$ such that $\sep({\A_w})\notin \SEP(\A_z[t^*])$ (otherwise,  by definition of separator, $\min(\chara({\A_w})\setminus \chara({\A_z}))$ would belong to $\chara({\A_z})$, which is impossible). So, at some given stage $t$, $M^*$ can choose the \emph{oldest} separator in $\SEP(\A_z[t])$, i.e., the separator that belongs to
\begin{equation}
\bigcap_{i\leq s\leq t}(\SEP(\A_z[s]))
\end{equation}
for the least $i$. 

\smallskip

To sum up, to learn any structure $\S\in \FrakA$ the learner $M^*$ does the following: At any given stage $t$, $M^*$ takes the output of $M(\S[t])$ as the current guess of the $\biEmbedsFin$-type of $\S$. Within the latter type $M^*$ considers only the structures whose separators are realized by $\S[t]$ and outputs the isomorphism type of the one with the oldest realized separator. By this procedure, $M^*$ $\Inf\Ex_{\cong}$-learns $\FrakA$.

\medskip

$(\Leftarrow):$ This implication can be proved via locking sequences. But instead of crudely applying Theorem \ref{theorem:lockingSequences}, we take here the opportunity of illustrating with some details  how to dynamically  build structures that serve as a counterexample to a given learning condition.

Assume that $\FrakA$ is not finitely separable. We show that $\FrakA\notin \Inf\Ex_{\cong}$. Towards a contradiction, suppose that $M$ learns $\FrakA$ and let $\A\in \FrakA$ be a limit of $\FrakA$. We construct an $\omega$-presentation $\B$ of some structure in $\FrakA$ that $M$ cannot learn. We start by constructing $\B$ as an $\omega$-presentation of $\A$ and we continue until $M$, on input $\B[s]$ for some $s$, outputs an index of $\A$. If this never happen, we obviously win: $M$ fails to learn an $\omega$-presentation of $\A$, and therefore $M$ does not $\Inf\Ex_{\cong}$-learn $\FrakA$. Otherwise, let $s$ be a stage such that $\A\cong M(\B[s])$. 
%and $char_{\B[s]}\subseteq c_{\A}$. 
Since $\A$ is a limit of $\FrakA$, there must be 
some $\S_0\in \FrakA$ such that $\S_0 \embedsFin \A$ $\S_0\not\cong\A$,  and $\chara({\B[s]})\subseteq \chara({\S_0})$. We now extend $\B[s]$ as an $\omega$-presentation of $\S_0$, with the caution of not expanding the equivalence classes already defined in $\B[s]$ (this can always be done since $\chara({\B[s]})\subseteq \chara({\S_0})$). We continue building $\B$ as an $\omega$-presentation of $\S_0$ until we find some stage $t$ such that $M$ correctly guesses our plan, i.e., $M(\B[t])\cong\S_0$. If this happens, we now go back to $\A$ and extend $\B[t]$ as an $\omega$-presentation of $\A$. We can do that because $\S_0 \embedsFin \A$. 

By iterating this reasoning, and possibly defining many $\S_i$'s, it is not hard to see that we can force $M$ to either fail at learning some structure in $\FrakA$ (this structure being either $\A$ or some $\S_i$)  or have infinitely many mind changes. If the latter case happens, the above caution of never expanding already defined $\B$-classes when we need to make $\B$ an $\omega$-presentation of some $\S_i$  guarantees that we eventually obtain an $\omega$-presentation of $\A$. 
\end{proof}

A nice consequence of our characterization theorem is that we can separate families of equivalence structures consisting of finitely many isomorphism types from those consisting of infinitely many isomorphism types by means of the following partial analogous of Compactness.

\begin{corollary}
The following hold.
\begin{enumerate}
\item If $\FrakA /_{\cong}$ is finite, then $\FrakA\in \Inf\Ex_{\cong}$ if and only if the structures of  $\FrakA$ are pairwise $\Inf\Ex_{\cong}$-learnable.
\item There is $\FrakA\notin \Inf\Ex_{\cong}$ with $\FrakA /_{\cong}$ is infinite such that, for all $\FrakB\subseteq \FrakA$, if $\FrakB /_{\cong}$ is finite, then $\FrakB\in \Inf\Ex_{\cong}$.
\end{enumerate}
\end{corollary}

\begin{proof}
$(1)$ follows immediately from item $(1)$ of Definition \ref{def:finiteseparability} and Theorem \ref{thm:characterization}. 

For $(2)$, let $\A_i = [j : 1 - \delta_{ij}]$, where $\delta_{ij} = \begin{cases}  1, &\text{if $i=j$;}\\ 0, &\text{otherwise,}\end{cases}$ and let $\FrakA=\set{\A_i}_{i\in\omega}$. We have that, for $i\neq j$, $\A_i$ and $\A_j$ can be finitely separated, by using $[j:1]$ as a separator for $\A_i$ and $[i:1]$ as a separator for $\A_j$. From item $(1)$ of the present corollary, it then follows that any subset of $\FrakA$ which consists of finitely many isomorphism types is $\Inf\Ex_{\cong}$-learnable. Yet, by Theorem \ref{thm:characterization},  $\FrakA\notin \Inf\Ex_{\cong}$, since $\A_i$ is a limit of $\FrakA\setminus \A_i$ for all $i$.
\end{proof}

%\begin{definition}
%A class of structures $\FrakA$ is called \emph{finitely separable} (f.s.) iff
%$$
%\forall A \in \FrakA \exists s_A \subseteq C_A \mbox{ finite } \forall B: s_A \subseteq C_B \Rightarrow B \not\embeds_{\mathrm{fin}} A.
%$$
%\end{definition}
%
%\begin{theorem}
%The following are equivalent.
%\begin{enumerate}
%	\item $\FrakA \in \Inf\Ex_{\cong}$.
%	\item $\FrakA$ is finitely separable.
%\end{enumerate}
%\end{theorem}
%
%There is $\FrakA$ such that any finite subset is learnable, but $\FrakA$ is not:
%$A_i = [j : 1 - \delta_{ij}]$.\footnote{Using the Kronecker delta function such that $\delta_{ij} = 1$ if $i=j$ and $\delta_{ij} = 0$ otherwise.} Not possible to learn them all, but possible to learn any finite subset. \Tcomment{Similar in result (but not in structure) to Example~\ref{ex:notLearningLimitInstances}.}

\subsection{Bounding the Complexity of the Learners}
\label{sec:complexityOfLearning}

The procedure described above for learning any finitely separable family is obviously noneffective. It is natural to ask how much information is needed to perform it. More generally, we want to investigate what can be learned by learners of fixed complexity.

\begin{definition}
A family $\FrakA$ of computably presentable structures is \emph{uniformly enumerable by $f$} if $f$ is a total computable function such that $\set{\M_{f(i)}}_{i\in\omega}$ is a one-to-one enumeration of all structures of $\FrakA$, up to isomorphism. 
\end{definition}

The next theorem shows that computable learners fail to learn all finitely separable families, even if we restrict to uniformly enumerable ones.

\begin{thm}
There is a uniformly enumerable  $\FrakA \in \Inf\Ex_{\cong}$  such that $\FrakA \not\in \mathbf{0}$-$\Inf\Ex_{\cong}$.
\end{thm}
\begin{proof}
We construct in a uniform way a   family $\FrakA=\set{\A_e,\B_e: e\in\omega}$ of computable structures which is finitely separable, but cannot be $\Inf\Ex_{\cong}$-learned by any computable learner.

\subsubsection*{Informal strategy}
For all $e$, we want to diagonalize against the learner $\varphi_e$ by building in stages a pair of equivalence structures $\A_e,\B_e$ that satisfy the following properties at all stages $s$,
\begin{itemize}
\item  there exists $n$ such that $\A_e[s]$ has isomorphism type $[e:n,1:t_0]$ and $\B_e[s]$ has isomorphism type $[e:n+1,1:t_1]$, 
\item and $\A_e[s]\subseteq \B_e[s]$.
\end{itemize}
The idea of the construction is that we want to force $\varphi_e$ to either fail at learning $\set{\A_e,\B_e}$ or have infinitely mind changes if attemping to learn a particular $\omega$-presentation of $[e:\omega, 1:\omega]$. To do so, we wait that $\varphi_e$ produces different outputs on $\A_e$ and $\B_e$, and while waiting we extend $\A_e$ and $\B_e$ as to make their isomorphism types, in the limit, to be respectively $[e:n,1:\omega]$ and $[e:n+1,1:\omega]$. 
If $\varphi_e(\A_e)\downarrow \neq \varphi_e(\B_e)\downarrow$ never happens, then $\varphi_e$ does not $\Inf\Ex_{\cong}$-learns $\set{\A_e,\B_e}$ since $\A_e \not\cong\B_e$. Otherwise, if at some stage we obtain $\varphi(\A_e)\downarrow \neq \varphi(\B_e)\downarrow$, we add to $\A_e$ two new equivalence classes of size $e$ and to $\B_e$ the same two equivalence classes and an additional one of size $e$. By iterating this reasoning we obtain that, if $\A_e$ and $\B_e$ have eventually isomorphism type $[e:\omega,1:\omega]$, then we can produce an $\omega$-presentation of $[e:\omega,1:\omega]$ on which $\varphi_e$ have infinitely many mind changes.

\subsubsection*{The construction}
We build in stages strings $\sigma$ and $\tau$ such that $\A_e=\cup_s \sigma_s$ and $\B_e =\cup_s \tau_s$, and a string $\nu$. During the construction we distinguish between expansionary and nonexpansionary stages and we make use of a counter $l$ that keeps track of the last expansionary stage. %At some stage $s$, we say that a number $x$ is \emph{fresh} if 

\medskip

\noindent\textbf{Stage $0$}: Let $\sigma_0=\tau_0=\nu_0=$ the empty string $\lambda$, and set $l:=0$.

\smallskip

\noindent\textbf{Stage $s+1$}: Assume that we have built $\sigma_s$ and $\tau_s$ with $\dom(\sigma_s)=\dom(\tau_s)$ and let $z_s$ be $\max(\dom(\sigma_s))+1$. We distinguish two cases.
\begin{enumerate}
\item If there exists $l\leq v\leq s$ such that $\varphi_{e,s+1}(\sigma_v)\downarrow \neq\varphi_{e,s+1}(\tau_v)\downarrow$, call $s+1$ an \emph{expansionary stage} and set $l:=s+1$. For $k\in\set{0,1,2}$, let  
\[
I_{k}=\set{z_{s}+ke,z_{s}+ke+1,\ldots,z_{s}+ke+e-1}.
\]
%\[
%B_{s+1}=\set{z_{s}+e,z_{s}+e+1,\ldots,z_{s}+2e-1},
%\]
%\[
%C_{s+1}=\set{z_{s}+2e,z_{s}+2e+1,\ldots,z_{s}+3e-1}.
%\]
Let $\sigma_{s+1}$ and $\tau_{s+1}$ be the following strings with domain $\set{x: 0\leq x\leq z_s+3e}$,

\[
\sigma_{s+1}(\langle x,y\rangle)=\begin{cases}
\sigma_{s}(\langle x,y\rangle) &\text{$x,y\in\dom(\sigma_s)$,}\\
1 &\text{$x=y \vee (\exists k \in \set{0,1})(x,y \in I_{k})$},\\
0 &\text{otherwise.}\\
\end{cases}
\]

and 

\[
\tau_{s+1}(\langle x,y\rangle)=\begin{cases}
\tau_{s}(\langle x,y\rangle) &\text{$x,y\in\dom(\tau_s)$,}\\
1 &\text{$x=y \vee (\exists k \in \set{0,1,2})(x,y \in I_{k})$},\\
0 &\text{otherwise.}\\
\end{cases}
\]

Finally, recall that $\varphi_{e,s+1}(\sigma_v)\downarrow \neq\varphi_{e,s+1}(\tau_v)\downarrow$. Without loss of generality, assume that $\varphi_{e,s+1}(\nu_{l})\neq\varphi_{e,s+1}(\sigma_v)$. Define $\nu_{s+1}=\sigma_{s+1}$ (the construction of $\sigma$ and $\tau$ guarantees that $\nu_{s+1}\supseteq \nu_{s}$).

\item Otherwise, let $\sigma_{s+1}\supseteq \sigma_s$ be the  only string such that $\dom(\sigma_{s+1})=\dom(\sigma_{s})\cup\set{z_s}$, $\sigma_{s+1}(\langle z_s,z_s\rangle)=1$, and  $\sigma_{s+1}(\langle x,y\rangle)=0$ if either $x=z_s$ or $y=z_s$. We define $\tau_{s+1}$ in the same way. Finally, let $\nu_{s+1}=\nu_{s}$.

\end{enumerate}

\subsubsection*{The verification}

It is immediate to see that $\A_e=\cup_s \sigma_s$ and $\B_e=\cup_s \tau_s$ are computable. We now distinguish two cases. 

First, suppose that there exist only finitely many expansionary stages. If so, it follows from the construction that there is a number $n$ such that $\A_e$ has isomorphism type $[e:n,1:\omega]$ and $\B_e$ has isomorphism type $[e:n+1,1:\omega]$. Towards a contradiction, suppose that $\varphi_e$ $\Inf\Ex_{\cong}$-learns $\set{\A_e,\B_e}$. If so, we have there are $\M_a\cong \A_e$ and $\M_b \cong \B_e$, with $\M_a \not\cong \M_b$, such that $\varphi_e$ on input $\A_e$ eventually converges to $a$, and on input $\B_e$ eventually converges to $b$. This means that there is a stage $s$ such that, for all $t\geq s$,
\[
\varphi_e(\sigma_t)=a\neq b=\varphi_e(\tau_t).
\]
But this immediately implies that there are infinitely many expansionary stages, contradicting the initial hypothesis. Hence, $\varphi_e$ does not $\Inf\Ex_{\cong}$-learn $\set{\A_e,\B_e}$. 

Second, suppose that there are infinitely many expansionary stages. From the construction, it follows that $\A_e$ and $ \B_e$ are isomorphic and they have isomorphism type $[e:\omega,1:\omega]$. The existence of infinitely many singletons in $\A_e$ and $\B_e$ comes from the fact that, if $s$ is an expansionary stage, then the number $z_s +3e$ is a singleton of $\A_e$ and a singleton of $\B_e$. Let $\C_e=\cup_s \nu_s$. We have that also $\C_e$ has isomorphism type $[e:\omega,1:\omega]$, since at any expansionary stage new $\C_e$-classes of size $e$ are defined and they never expand later. If $\varphi_e$ $\Inf\Ex_{\cong}$-learns $\set{\A_e,\B_e}$, then $\varphi_e$ has to learn $\C_e$ as well, since $\C_e$ is an $\omega$-presentation of $\A_e$ and $\B_e$. So, there must be some $\M_c \cong \C_e$ such that $\varphi_e$ on input $\C_e$ eventually converges to $c$. This means that there is a stage $s$ such that, for all $t\geq s$, $\varphi_e(\nu_t)=c$. But this is impossible, since $\nu$ is constructed in such a way that, if $t_1$ and $t_2$ are consecutive expansionary stages, then $\varphi_e(\nu_{t_1})\neq \varphi_e(\nu_{t_2})$. We conclude that $\varphi_e$ does not $\Inf\Ex_{\cong}$-learn $\C_e$, and thus does not learn $\set{\A_e,\B_e}$.
\end{proof}

 Next, we want to consider the computational complexity of learning. 
Although computable learners do not have enough power to grasp all finitely separable families that are uniformly enumerable, two jumps suffice to learn equivalence structures with no infinite classes. 

\begin{thm}
Let $\FrakA$ be uniformly enumerable and such that no equivalence structure from $\FrakA$ has infinite equivalence classes. The following are equivalent. 

\begin{enumerate}
\item $\FrakA$ is finitely separable.
\item  $\FrakA \in \Inf\Ex_{\cong}$.
\item $\FrakA \in \mathbf{0''}\mbox{-}\Inf\Ex_{\cong}$
\end{enumerate}
\end{thm}

\begin{proof}
``$(1) \Leftrightarrow (2)$'' is the content of Theorem \ref{thm:characterization}. The direction ``$(3) \Rightarrow (2)$'' is trivial. We prove  ``$(2) \Rightarrow (3)$''. 

Assume that $\FrakA\in\Inf\Ex_{\cong}$ is uniformly enumerable by $f$. Let $M$ and $M^*$ be defined as in the proof of Theorem \ref{thm:characterization}; in particular, recall that $M^*$ $\Inf\Ex_{\cong}$-learns $\FrakA$. We follow the steps of the  proof of Theorem \ref{thm:characterization} to show that there are $\mathbf{0}''$-computable $N$ and $N^*$ such that $N$ eventually coincides with $M$ and $N^*$ eventually coincides with $M^*$. Fix $\A \in \FrakA$.

Note that, for all $s$, the following set is $\mathbf{0}''$-computable
\[
X_s=\set{i\leq s: \A[s]\embedsFin \M_{f(i)} \wedge (\forall j\leq s)(\M_{f(j)} \not\embedsFin \M_{f(i)}) }.
\]
  This follows from the fact that, given any two computable structures  $\A,\B\in\mathbb{E}$, $\A \embedsFin \B$  holds if and only if $(\forall s \exists t)(\A[s] \embeds \B[t])$, and  $\mathbf{0}''$ can decide this $\Pi^0_2$ formula.  Define $N(n)=\min(X_n)$. By reasoning as in the proof of Theorem \ref{thm:characterization}, it is not difficult to see that   $\bigcap_{n\in\omega} \M_{N(n)}$ is contained in the $\biEmbedsFin$-type of $\A$.

We now observe that the family of separators of computable structures in $[\M_{N(n)}]_{\biEmbedsFin}$ is uniformly c.e.\ in $\mathbf{0}''$. Indeed, given $\M_{f(z)} \in [\M_{N(n)}]_{\biEmbedsFin}$, by definition of separator we obtain  
\[
x \in \sep({\M_{f(z)}}) \Leftrightarrow (\exists j)[\M_{f(j)} \biEmbedsFin \M_{f(z)} \wedge  x=\min(\chara({\M_{f(z)}})\setminus \chara({\M_{f(j)}}))].
\]

Since all the $\M_{f(i)}$'s have no infinite equivalence classes,  $\chara({\M_{f(z)}})\setminus \chara({\M_{f(j)}})$ is a $\Sigma^0_2$ set. Moreover, as already observed,  to ask whether  $\M_{f(j)} \biEmbedsFin \M_{f(z)}$ holds is a $\Pi^0_2$ question. So, the overall condition is $\Sigma^0_3$, and thus c.e.\ in $\mathbf{0}''$. We can conveniently approximate such separators as follows. At stage $s$, $N^*$ chooses the structure in  $\mathcal{F}_n=[\M_{N(n)}]_{\biEmbedsFin}\cap \set{\M_{f(i)}}_{i\leq n}$ whose separator, restricted to the elements of $\mathcal{F}_n$, is  the oldest realized by $\A_n$ (notice that to check whether a given separator is realized by $\A[s]$ can be done effectively, since any separator is finite). $N^*$ so defined is $\mathbf{0}''$-computable and eventually outputs the same value of $M^*$, hence $N^*$ $\Inf\Ex_{\cong}$-learns $\FrakA$.
  \end{proof}

The next question is left open.

\begin{question}
Is there a uniform enumerable family $\FrakA\in \Inf\Ex_{\cong} \setminus 0'$-$\Inf\Ex_{\cong}$?
\end{question}

\section{Related Learning Settings}

In this section we consider several learning criteria related to $\Inf\Ex_{\isom}$ and show how they compare. This provides us with many examples of families of structures which are learnable in one setting, but not another. 

\smallskip

First, we characterize which finite collections of equivalence structures can be finitely learned.

\begin{thm}\label{thm:FinLearningOfFiniteClasses}
Let $\FrakA$ be a family of equivalence relations such that $\FrakA /_{\cong}$ is finite. The following are equivalent.
\begin{enumerate}
	\item $\FrakA \in \Inf\Fin_{\isom}$.
	\item $\forall \A,\B \in \FrakA: \A \embedsFin \B \Rightarrow \A \isom \B$.
	\item\label{it:finLearningPartThree} $\FrakA /_{\mathclose{\isom}}$ is an anti-chain with respect to $\embedsFin$.
\end{enumerate}
\end{thm}
\begin{proof}
Suppose first $\FrakA \in \Inf\Fin_{\isom}$, witnessed by some learner $M$. Let $\A,\B \in \FrakA$ with $\A \embedsFin \B$. Let $\sigma$ describe a part of $\A$ such that $M$ makes a (correct) conjecture for $\A$ on $\sigma$. Without loss of generality, suppose that all mentioned items of $\sigma$ which are equivalent in $\A$ are mentioned equivalent in $\sigma$. We have $\A_\sigma$ (the finite structure coded by $\sigma$) is now a finite substructure of $\A$, so, by supposition, $\A_\sigma \embeds \B$. This shows that $\sigma$ can be extended to an informant $I$ for (an isomorphic copy of) $\B$. Since $M$ already makes an output on $\sigma$, the output of $M$ on $I$ is $M(\sigma)$. This shows that $M(\sigma)$ is a correct conjecture for $\B$, which gives
$$
\A \isom \M_{M(\sigma)} \isom \B.
$$

Suppose now $\forall \A,\B \in \FrakA: \A \embedsFin \B \Rightarrow \A \isom \B$. For each $\A \in \FrakA$, let $K(\A)$ be a finite substructure of $\A$ such that, for all $\B \in \FrakA$ with $\A \not\isom \B$, $K(\A) \not\embeds \B$ (which exists since $\FrakA$ contains only finitely many $\isom$-types). Then, for all $\A, \B \in \FrakA$ with $\A \not\isom \B$ we have that neihter $K(\A)$ is embeddable into $K(\B)$ nor vice versa. This is equivalent to \ref{it:finLearningPartThree}.

Using such a list of pairwise incomparable finite substructures, we can define a learner $M$ such that, on input $\sigma$, $M(\sigma)$ is a conjecture for $\A \in \FrakA$ if $K(\A) \embeds \A_\sigma$ with an embedding which may not map elements from different equivalence classes of $K(\A)$ into elements of different equivalence classes of $\A_\sigma$, unless $\sigma$ explicitly contains the information that these equivalence classes are different. If no such $\A$ exists (which would be necessarily unique), then $M(\sigma) = ?$. Clearly, this learner $\Inf\Fin_{\isom}$-learns $\FrakA$.
\end{proof}
Note that the above proof extends to any equivalence relation on structures in place of $\isom$.

\smallskip

Next we show that that there is a class of two structures which is learnable up to bi-embeddability, but not up to isomorphism, showing Example~\ref{ex:biEmbedsNotIsom} formally.

\begin{thm}\label{thm:computablyBiembedButNotIsomLearnableShort}
The following holds
\[
\Inf\Ex_{\cong} \subset \Inf\Ex_{\approx}.
\]
\end{thm}

\begin{proof}
It is obvious that any family that is learnable up to isomorphism is also learnable up to bi-embeddability. This show that $\Inf\Ex_{\cong}\subseteq \Inf\Ex_{\approx}$. To see that $\Inf\Ex_{\approx}\nsubseteq\Inf\Ex_{\cong}$, let $\A$ be a structure of type $[5:\omega]$ and $\B$ a structure of type $[5:\omega,2:1]$. We have that $\A\approx \B$ and thus, trivially, $\set{\A,\B}\in \Inf\Ex_{\approx}$ by a learner that always conjectures the isomorphism type of $\A$. On the other hand, $\A$ is clearly a limit of $\B$ (see Definition \ref{def:finiteseparability}), and so $\set{\A,\B}$ is not finitely separable. By Theorem \ref{thm:characterization}, this means that $\set{\A,\B}\notin\Inf\Ex_{\cong}$.
\end{proof}

\smallskip

Finally we consider learning without negative information, that is, learning from text rather than from informant. We establish with Theorem~\ref{thm:textEquivalentToInformant} that these two settings are the same as long as only families of structures without infinite equivalence classes are considered; Theorem~\ref{thm:textNoteEquivalentToInformant} then shows that, for families of structures with infinite equivalence classes, we get a separation of learning power.

\begin{thm}\label{thm:textEquivalentToInformant}
Let $\sim$ be any equivalence relation on $\structs$.
Let $\FrakA$ be such that none of the elements of $\FrakA$ has an infinite equivalence class. The following are equivalent.
\begin{enumerate}
	\item $\FrakA$ is $\Inf\Ex_{\sim}$-learnable.
	\item $\FrakA$ is $\Text\Ex_{\sim}$-learnable.
\end{enumerate}
\end{thm}
\begin{proof}
The direction ``(2) $\Rightarrow$ (1)'' is immediate. Regarding ``(1) $\Rightarrow$ (2)'', let $M$ be a strong informant locking $\Inf\Ex_{\sim}$-learner for $\FrakA$ (see Theorem~\ref{theorem:textStronglyLocking}). We now descirbe how $\FrakA$ can be learned from text. Given an initial part of a text $\sigma$ we reorder this information as follows. First comes all positive information regarding the ``first'' equivalence class, defined as the equivalence class of $0$. Then comes the positive information regarding the second smallest equivalence class, the class containing $1$ (unless $1$ was already covered by the equivalence class of $0$, then we take $2$ and so on). Then comes all the negative information between the first and second class (this is not present in $\sigma$, just assumed). Then all positive information of the third class, followed by the negative information between the third class on the one hand and the first and second class on the other hand, and so on, until all elements mentioned in $\sigma$ are covered. The resulting partial informant we call $\overline{\sigma}$. We can now define a learner $M'$ learning from text as $M'(\sigma) = M(\overline{\sigma})$. 

For any given $\A \in \FrakA$, there is an informant $I$ which presents the data in the form described above. Let $k$ be the largest size for which there are infinitely many equivalence classes of that size in $\A$. Since $M$ is strong informant locking, there is $n$ such that $I[n]$ is a strong locking sequence for $M$ on $\A$. For any Text $T$, there is now an $n'$ such that $T[n']$ contains (a) all the positive information contained in $I[n]$; (b) all positive information about the equivalence classes of $\A$ that are larger than $k$; and (c) all positive information about all equivalence classes containing elements that are numerically smaller than any of the elements mentioned in (a) or (b). This gives that, for all $m > n'$, we have that $\overline{T[n']}$ is extensible to an isomorphic copy of $\A$. Thus, using that $I[n]$ is a strong locking sequence, for all $m > n'$ we have that $M'(T[m]) = M(I[n])$ as desired.
\end{proof}

In contrast to this result, infinite equivalence classes quickly lead to differences between text- and informant-learning.

\begin{thm}\label{thm:textNoteEquivalentToInformant}
Let $\FrakA = \set{[\omega: 1], [\omega: 2]}$. Then we have the following.
\begin{enumerate}
	\item $\FrakA$ is $\Inf\Ex_{\isom}$-learnable.
	\item $\FrakA$ is not $\Text\Ex_{\biEmbeds}$-learnable.
\end{enumerate}
\end{thm}

\begin{proof}
Regarding (1), the learner conjectures $[\omega:1]$ while no negative data was given, and $[\omega:2]$ afterwards.

Regarding (2), suppose by contradiction that a $\Text\Ex_{\biEmbeds}$-learner $M$ for $\FrakA$ exists. Using Theorem~\ref{theorem:lockingSequencesBiembed}, there is a strong locking sequence $\sigma$ of $M$ on $[\omega: 1]$ with corresponding embedding $f$. We can extend $\sigma$ to a sequence $\tau$ for $[\omega: 2]$ for which $M(\tau)$ is a correct conjecture for $[\omega: 2]$. Since we can extend $f$ trivially to $\A_{\tau}$, we get a contradiction.
\end{proof}

\section{Relation to Language Learning}
\label{sec:languageLearning}

In this section we compare $\Inf\Ex_{\isom}$-learning to the cardinal learning criterion of formal languages, $\Text\Ex$-learning. The formal definition of this criterion can be found for example in \cite{Jai-Osh-Roy-Sha:b:99:stl2}, we will here use the following characterization (see \cite{Ang:j:80:lang-pos-data}). A collection $\CalL$ of formal languages is $\Text\Ex$-learnable iff, for all $L \in \CalL$ there is a finite $D \subseteq L$ such that, for all $L' \in \CalL$ with $D \subseteq L' \subseteq L$ we have $L' = L$. Intuitively, $D$ signals $L$ as the minimal extrapolating target. We call this characterization \emph{Angluin's tell-tale criterion} and the finite sets $D$ are called \emph{tell-tales}.

We are now interested in somehow mapping learning tasks for $\Inf\Ex_{\isom}$-learning into learning tasks for $\Text\Ex$-learning. The next theorem shows that there cannot be a mapping taking an (isomorphism class of a) structure to a single language in order to embed $\Inf\Ex_{\isom}$-learning into $\Text\Ex$-learning.
\begin{thm}\label{thm:isomLearningToLanguageLearningNotSingle}
Let a mapping $\Theta$ be given which takes a structure and returns a language such that $\forall \A,\B \in \structs: \Theta(\A) = \Theta(\B) \Leftrightarrow \A \isom \B$. Then there is a class of structures $\FrakA$ such that
$$
\FrakA \not\in [\Inf\Ex_{\isom}] \mbox{ and } \set{\Theta(\A) \mid \A \in \FrakA} \in [\Text\Ex].
$$
\end{thm}
\begin{proof}
Let $\FrakA$ consist of the $\isom$-closure of $[5:\omega,2:1]$ and $[5:\omega]$. We know that $\FrakA \not\in [\Inf\Ex_{\isom}]$. We have $\set{\Theta(\A) \mid \A \in \FrakA}$ contains only two languages, which is trivially $\Text\Ex$-learnable by Angluin's tell-tale condition.
\end{proof}

In order to bypass this phenomenon, instead of associating only \emph{one} language with any isomorphism type of structures, we can associate an \emph{infinite} set of languages. The next theorem shows that this way we can derive an embedding. We will use the definition of a \emph{finite permutation}, which is any permutation $\pi$ of $\natnum$ such that, for all but finitely many $x \in \natnum$, $\pi(x) = x$. 

\begin{thm}\label{thm:isomLearningToLanguageLearning}
For any structure $\A \in \structs$ we define a set of languages as follows. Let $f: \natnum \cup \{\omega\} \rightarrow \natnum \cup \{\omega\}$ represent the isomorphism type of $\A$, that is, $[f] = [\A]_{\isom}$. Let $g_A$ be any computable function (where we allow $\omega$ as a special symbol output) such that, for all $a \in \natnum_+ \cup \{\omega\}$, $f(a) = |\set{i \in \natnum \mid g_A(i) = a}|$.\footnote{Intuitively, each equivalence class of $\A$ is associated with some $i\in \natnum$, and $g_A(i)$ is the size of this equivalence class.} For any function $h \colon \natnum \rightarrow \natnum \cup \{\omega\}$ we let $L(h) = \set{\pair{i,j} \mid j < h(i)}$. We let $\CalL_{\A} = \set{L(g_A \circ \pi) \mid \pi \mbox{ finite permutation}}$.

Then, for all $\FrakA \subseteq \structs$ closed under $\isom$, the following are equivalent.
\begin{enumerate}
	\item $\FrakA$ is $\Inf\Ex_{\isom}$-learnable.
	\item $\left(\bigcup_{\A \in \FrakA} \CalL_{\A}\right)$ is $\Text\Ex$-learnable.
\end{enumerate}
\end{thm}
\begin{proof}
Regarding ``$\Rightarrow$'', let an $\Inf\Ex_{\isom}$-learner $M'$ for $\FrakA$ be given. Now we construct a learner $M$ which is given some sequence of data $\sigma$. From this sequence the learner constructs a sequence $\sigma'$ for $M$ so that $\sigma'$ encodes the finite structure on the elements mentioned in $\sigma$, where $\langle i,j \rangle$ is equivalent to $\langle i',j' \rangle$ iff $i = i'$. Let $\A \in \structs$ be the structure conjectured by $M'$ on $\sigma'$. In some canonical listing of all finite permutations, find the minimal finite permutation $\pi$ such that $\sigma$ is consistent with $L(g_A \circ \pi)$ and conjecture this language.

Let now $\A \in \FrakA$ and $L \in \CalL_{\A}$ and a text $T$ for $L$ be given. Let $n_0$ be large enough such that $M'$ on $\sigma'$ is converged on the structure $\A$. Let $\pi$ be the minimal finite permutation such that $L = L(g_A \circ \pi)$. It now suffices to show that, for all $\pi' < \pi$ in the canonical listing, $L(g_A \circ \pi')$ is inconsistent with some finite part of $L$. Let $\pi' < \pi$ be given, and let $i$ be minimal such that $g_A(\pi(i)) > g_A(\pi'(i))$; such an $i$ has to exist, since either we have for all $g_A(\pi(i)) = g_A(\pi'(i))$, in which case $\pi$ was not chosen minimal as required, or there is a difference, in which case a difference has to be found both ways, since we only consider permutations. This shows that $\langle i, g_A(\pi'(i)) \rangle \in L \setminus L(g_A \circ \pi')$ as desired.

Regarding ``$\Leftarrow$'', %let a $\Text\Ex$-learner $M$ for $\left(\bigcup_{A \in \FrakA} \CalL_A\right)$ be given. 
we construct a learner $M'$ for $\FrakA$ as follows. Given an initial part $\sigma'$ of an informant, we assign each element mentioned in $\sigma'$ in order of (first) appearance to some $i \in \natnum$: any element $x$ known to be equivalent to a previously assigned element $y$ is assigned to the same $i$ as $y$ was assigned; any element not thus assigned is assigned to the smallest $i$ not yet used as an assignment. Let $g(i)$ be the number of elements assigned to $i$. Then $M'$ checks whether there is a minimal superset of the finite set $\set{\langle i, j \rangle \mid j < g(i)}$ in $\left(\bigcup_{\A \in \FrakA} \CalL_{\A}\right)$. If not, the conjecture is $?$. Otherwise, let $L(g_A \circ \pi)$ be the minimal superset and $M'$ conjectures (a canonical index for) $\A$.

Let now $\A \in \FrakA$ be given and let $I$ be an informant for $\A$. Note that, once $M'$ conjectures an index for $\A$, $M'$ has converged to $\A$. Otherwise a later conjecture for $B$ would imply the existence of a finite permutation $\pi$ such that $L(g_B \circ \pi)$ is consistent with known data; this would have been consistent already at the earlier point when the conjecture was $\A$, which would be a contradiction.

Let $g$ be as produced in the limit by the construction of $M'$ on $I$. Then there is a permutation $\pi$ of $\natnum$ (not necessarily finite) such that $g = g_A \circ \pi$. Using Angluin's tell-tale condition, let a tell-tale $D$ for $L(g_A)$ be given. This implies that once at least the set $\set{ \langle \pi^{-1}(i),j \rangle \mid \langle i,j \rangle \in D}$ has been used in the construction of $M'$, the output of $M'$ will be a conjecture for $\A$ as desired.
\end{proof}

\medskip

\bibliographystyle{plain}

\clearpage
\appendix

\section{Locking Sequences}
\label{sec:lockingSequences}

\newcommand{\ext}{\mathrm{ext}}
\newcommand{\exts}{\mathrm{ext}_{\mathrm{seq}}}
\newcommand{\exte}{\mathrm{ext}_{\mathrm{emb}}}
\newcommand{\exth}{\mathrm{ext}_{\mathrm{hom}}}

Locking sequences are a powerful tool for many theorems in the area of learning theory. We first explore the concept of \emph{weak locking sequences} which corresponds to standard locking sequences in learning languages (and the proof is standard, see \cite{Jai-Osh-Roy-Sha:b:99:stl2}). Then we consider a strong variant where locking happens for a much larger class of possible extensions.

\subsection{Weak Locking Sequences}

We start by considering the classic locking sequences and give a straightforward generalization to arbitrary equivalence relations on the structures.

\begin{definition}
Let $M$ be a learner and $\A$ a structure. We say that a sequence $\sigma$ describing a finite part of $\A$ is a \emph{weak locking sequence of $M$ on $\A$} iff, for every $\tau \supseteq \sigma$ describing a finite part of $\A$, we have $M(\sigma) = M(\tau)$. We distinguish between weak \emph{informant} locking sequences which consist of negative and positive data (and the extensions $\tau$ are allowed positive and negative data), and weak \emph{text} locking sequence which consist of positive data only (and the extensions $\tau$ are also only allowed positive data).
\end{definition}

\begin{thm}\label{theorem:lockingSequences}
Let $\sim$ be any equivalence relation on $\structs$.
Suppose $M$ $\Inf\Ex_\sim$-learns a structure $\A \in \structs$. Let any sequence $\sigma_0$ be given which describes a finite part of $\A$. Then there is a finite sequence $\sigma \supseteq \sigma_0$ such that $\sigma$ is a weak informant locking sequence of $M$ on $\A$. Furthermore, $\M_{M(\sigma)} \sim \A$.
\end{thm}
\begin{proof}
Let $\sigma_0$ be given and suppose no such $\sigma$ exists. Thus we can fix, for any $\sigma \supseteq \sigma_0$ describing a finite part of $\A$, $\ext(\sigma)$ as an extension of $\sigma$ which gives a mind change. Let $I_0$ be an informant for $\A$.

We define inductively
\begin{align*}
\forall i: \sigma_{2i+1} &= \sigma_{2i} \diamond I_0(i);\\
\forall i: \sigma_{2i+2} &= \ext(\sigma_{2i+1}).
\end{align*}
Finally, let $I = \bigcup_i \sigma_i$. Then $I$ is an informant for $\A$ on which $M$ does not converge, a contradiction. The ``furthermore'' clause of the statement follows since $M$ needs to converge to a $\sim$-correct conjecture on any informant.
\end{proof}

The first application of the locking sequences theorem is to prove a normal form for learners, which we will define next. This normal form can be very convenient for proofs.

\begin{definition}
Let $M$ be a learner and $\A$ a structure. We call $M$ \emph{locking} on $\A$ iff, for all informants $I$ for $\A$, there is $n$ such that $I[n]$ is a weak locking sequence for $M$ on $\A$.
Let $\sim$ be any equivalence relation on $\structs$. We call $M$ \emph{locking} iff $M$ is locking for all $\A \in \Inf\Ex_\sim$. Again we distinguish between \emph{informant locking} and \emph{text locking}.
\end{definition}

\begin{thm}\label{theorem:informantStronglyLocking}
Let $\sim$ be any equivalence relation on $\structs$ and let $M$ be a $\Inf\Ex_\sim$-learner. 
Then there is an informant locking $\Inf\Ex_\sim$-learner $M'$ which learns at least all structures learned by $M$.
\end{thm}
\begin{proof}
For an input informant $I$ we define inductively $\sigma_0$ as the empty sequence and, for all $n$, $\sigma_{n+1}$ as $\sigma_n$ concatenated with any information $\tau$ of $I[n]$ such that $M(\sigma_n) \neq M(\sigma_n \diamond \tau)$. Note that, for all $\A$ which are $\Inf\Ex_\sim$-learned by $M$, we get that $(\sigma_n)_n$ converges. Further note that $\sigma_n$ can be computed from $I[n]$.

We now define a learner $M'$ such that $M'(I[n]) = M(\sigma_n)$. Clearly, $M'$ $\Inf\Ex_\sim$-learns any structure $\Inf\Ex_\sim$-learned by $M$. Furthermore, $M'$ is strongly locking, since $(\sigma_n)_n$ converges.
\end{proof}

By the analogous proof we get the analogous theorem for texts.
\begin{thm}\label{theorem:textStronglyLocking}
Let $\sim$ be any equivalence relation on $\structs$. 
Let $M$ be a $\Text\Ex_\sim$-learner. Then there is a text locking $\Text\Ex_\sim$-learner $M'$ which learns at least all structures learned by $M$.
\end{thm}

\subsection{Strong Locking Sequences}

The above theorems about locking sequences are already rather powerful and used extensively in the case of language learning. However, since a learner needs to learn \emph{any} $\omega$-presentation of structure, we can get even stronger locking than the ones given by Theorem~\ref{theorem:lockingSequences}. Here we do not only lock on extensions from the chosen concept, but also on other sequences which can be considered equivalent. This was already proven by \cite{Mar-Osh:b:98}.

Recall that, for any $\sigma$, $\sigma$ encodes a finite equivalence relation on the elements mentioned either positively or negatively by using the transitive closure of all positively mentioned pairs and assuming all other relations to be negative. This finite equivalence relation we call $\A_\sigma$.

\begin{definition}
Let $M$ be a learner and $\A$ a structure. We say that a sequence $\sigma$ describing a finite part of $\A$ is a \emph{strong locking sequence of $M$ on $\A$} iff there is an embedding $f$ embedding $\A_\sigma$ into $\A$ such that, for any $\tau$ extending $\sigma$ and any $g$ embedding from $\A_\tau$ to $\A$ which extends $f$, $M(\sigma) = M(\tau)$.  We distinguish between strong \emph{informant} locking sequences which consist of negative and positive data (and the extensions $\tau$ are allowed positive and negative data), and strong \emph{text} locking sequence which consist of positive data only (and the extensions $\tau$ are also only allowed positive data).
\end{definition}

\begin{thm}[Martin and Osherson~\cite{Mar-Osh:b:98}]\label{theorem:lockingSequencesBiembed}
Let $\FrakA$ be a set of equivalence structures and let $\A \in \FrakA$. Suppose $M$ $\Inf\Ex_\biEmbeds$-learns $\FrakA$ and let any sequence $\sigma_0$ be given which describes a finite part of $\A$. Then there is a strong informant locking sequence $\sigma \supseteq \sigma_0$ of $M$ on $\A$.
\end{thm}
\begin{proof}
Suppose no such $\sigma,f$ exist. Thus, for any given $\sigma,f$ with $\sigma \supseteq \sigma_0$, let $\sigma',f'$ be first extensions $\tau,g$ found in a dovetailing search which lead to a mind change. Furthermore, for any $\sigma \supseteq \sigma_0$ and any embedding $f$ of $\A_\sigma$ into $\A$, we let $\ext(\sigma,f,i)$ denote a pair $(\tau,g)$ such that $\tau$ extends $\sigma$ and $\tau$ labels at least all pairs from $\set{0,\ldots, i}^2$, the image of $g$ contains $\set{0,\ldots,i}$ and $g$ is an embedding of $\A_\tau$ into $\A$; we denote $\tau$ by $\exts(\sigma,f,i)$ and $g$ by $\exte(\sigma,f,i)$.

Let $f_0$ be any embedding of $\A_\sigma$ into $\A$. We define inductively
\begin{align*}
\forall i: \sigma_{i+1} &= \exts(\sigma_i',f_i',i);\\
\forall i: f_i &= \exte(\sigma_i',f_i',i).
\end{align*}
Finally, let $I = \bigcup_i \sigma_i$, $\B$ the structure described by $I$, and $f = \bigcup_i f_i$. Then $f$ witnesses $\A \isom \B$ and $I$ witnesses that $M$ does not learn $\B$, a contradiction.
\end{proof}

We can now have the analogous definition about locking learners as for weak locking sequences.

\begin{definition}
Let $M$ be a learner and $\A$ a structure. We call $M$ \emph{strong-locking} on $\A$ iff, for all informants $I$ for $\A$, there is $n$ such that $I[n]$ is a strong locking sequence for $M$ on $\A$.
Let $\sim$ be any equivalence relation on $\structs$. We call $M$ \emph{strong locking} iff $M$ is locking for all $\A \in \Inf\Ex_\sim$. Again we distinguish between \emph{strong informant locking} and \emph{strong text locking}.
\end{definition}

\begin{thm}\label{theorem:informantStrongLocking}
Let $\sim$ be any equivalence relation on $\structs$ and let $M$ be a $\Inf\Ex_\sim$-learner. 
Then there is a strong informant locking $\Inf\Ex_\sim$-learner $M'$ which learns at least all structures learned by $M$.
\end{thm}
\begin{proof}
Analogous to the proof of Theorem~\ref{theorem:informantStronglyLocking}.
\end{proof}

\bigskip

\end{document}